\documentclass[11pt,reqno]{amsart}
\usepackage{amsmath}
\usepackage{etoolbox}
\usepackage{amscd, amsfonts, amssymb, graphicx, color}
\usepackage{url}
\usepackage{tikz-cd}
\usepackage{hyphenat}
\usepackage{mathtools}
\usepackage[colorlinks=true,pagebackref=true]{hyperref}
\usepackage{titlesec}
\usepackage{array}
\usepackage{caption}
\usepackage{geometry}
\titleformat{\section}
{\normalfont\bfseries\large\centering}{\thesection}{1em}{}
\makeatletter
\patchcmd\maketitle
{\uppercasenonmath\shorttitle}
{}
{}{}
\patchcmd\maketitle
{\@nx\MakeUppercase{\the\toks@}}
{\the\toks@}
{}
{}{}
\patchcmd\@settitle{\uppercasenonmath\@title}{\Large}{}{}
\patchcmd\@setauthors
{\MakeUppercase{\authors}}
{\authors}
{}{}
\makeatother
\newtheorem{theorem}{Theorem}[section]

\newtheorem{lemma}{Lemma}[section]

\newtheorem{example}{Example}[section]

\newtheorem{Proof of Theorem}{Proof}
\newtheorem{proposition}[theorem]{Proposition}

\hypersetup{urlcolor=blue, citecolor=red, linkcolor= blue}

\newcommand{\norm}[1]{\left\lVert#1\right\rVert}

\makeatletter
\renewcommand\subsection{\@startsection{subsection}{2}%
	\z@{.7\linespacing\@plus\linespacing}{.5\linespacing}%
	{\normalfont\bfseries}}
\makeatother
\begin{document}
		\title[Spectral Properties of the Compact Rhaly and Compact Generalised Ces{\`a}ro Operators]
		{{Spectral Properties of the Compact Rhaly and Compact Generalised Ces{\`a}ro Operators on Weighted $c_0$ Spaces}}
		
		\author[J. Rani, A. Patra  ] { {\large Jyoti Rani}$^{1}$, {\large Arnab Patra }$^{2}$}
		\address{$^{[1]}$ Department of Mathematics, Indian Institute of Technology Bhilai, Bhilai 491002, India}
		\email{\url{jyotir@iitbhilai.ac.in}}
		
		\address{$^{[2]}$ Department of mathematics, Indian Institute of Technology Bhilai , Bhilai 491002, India.
		}
		\email{\url{arnabp@iitbhilai.ac.in}}
		\date{\today}
		%\date{June 24, 2019}
		\maketitle		
\section{Abstract}
			In this article, we conduct a comprehensive study on the continuity, compactness, and spectral properties of Rhaly operators and generalized Cesàro operators, acting on weighted null sequence spaces. We determine the point spectrum, continuous spectrum, and residual spectrum for compact Rhaly operators and compact generalized Cesàro operators. Additionally, we explore Goldberg's classifications of Rhaly operators over weighted null sequence spaces.

		%\keywords{Rhaly operators, Generalised Ces{\`a}ro operators, fine spectrum, sequence spaces.}
		%
		%
		%
		\section{Introduction}
		
		Let $a=\{a_n\}$ be a sequence of real or complex numbers. The lower triangular matrix $R_a$ was introduced by Rhaly \cite{rhaly} and is defined as
		\[
		R_a=
		\begin{pmatrix}
			a_1 & 0 & 0 & 0 & 0 & \cdots \vspace{1.5mm}\\
			a_2 & a_2 & 0 & 0 & 0 & \cdots\vspace{1.5mm}\\
			a_3 & a_3 & a_3 & 0 & 0 & \cdots\vspace{1.5mm}\\
			a_4 & a_4 & a_4 & a_4 & 0 & \cdots\vspace{1.5mm}\\
			\vdots & \ddots & \ddots&\ddots & \ddots & \ddots
		\end{pmatrix}
		.\]
		The Rhaly operator is the operator that $R_a$ represents. The Ces{\'a}ro operator $C$ can be obtained by choosing $a_n = \frac{1}{n}$, $n \in \mathbb{N}$. Many reserachers have examined the spectrum of Rhaly operators defined over various classical sequence spaces, such as $c_0$ \cite{rhaly_c0_1,rhaly_c0_3}, $\ell_p$ \cite{rhaly_lp}, $bv_0$ \cite{rhaly_bv0}, etc. In this study, we investigated the spectrum characteristics of comapct Rhaly operators on weighted null sequence spaces. The spectral characteristics of discrete generalized Ces{\`a}ro operators defined over weighted null sequence spaces have been explored sequentially. 
		
		In 1982, Rhaly \cite{rhaly1982discrete} introduced the discrete generalized Cesàro operator $C_t$. The discrete generalized Ces{\`a}ro operator $C_t$ is defined as
		
		\[
		C_t=
		\begin{pmatrix}
			1 & 0 & 0 & 0 & \cdots \vspace{1.5mm}\\
			\frac{t}{2} & \frac{1}{2} & 0 & 0 &\cdots\vspace{1.5mm}\\
			\frac{t^2}{3} & \frac{t}{3} & \frac{1}{3} & 0 & \cdots\vspace{1.5mm}\\
			\vdots & \vdots & \vdots&\vdots & \ddots
		\end{pmatrix}
		.\]
		Numerous studies have been conducted in the literature about the spectral properties of the discrete generalized Cesàro operator $C_t$. Readers are directed to \cite{albanese2023spectral, curbera2022fine, albanese2024spectra} for a more thorough analysis.
		Take $t=1$ to obtain the Cesàro operator $C$. The spectral characteristics of the Cesàro operator have been examined by numerous researchers. We direct the reader to \cite{cesaro3,qcesaro3, cesaro_c0_c} for contributions regarding spectral characteristics of Cesàro operators.

The remaining sections of this work are organized as follows. Section 2 includes several notations and relevant results. Section 3 discusses a sufficient condition for the Rhaly operator's compactness, as well as the spectral properties of compact Rhaly operators defined over weighted $c_0$ space. In Section 4, we present some findings on Goldberg's classifications of the spectrum of the Rhaly operator defined over weighted $c_0$ spaces. Section 5 contains boundedness and compactness criteria for Ces{\'a}ro operators, as well as spectral subdivisions of generalized Ces{\'a}ro operators on weighted $c_0$ spaces.
		
		\section{Preliminaries and Notations}
		The natural number set $\mathbb{N}$ is the index for infinite sequences and matrices that are the subject of this article. Let $X$ and $Y$ be complex Banach spaces. The space of all bounded linear operators from $X$ to $Y$ is represented by the set $\mathcal{B}(X,Y)$, while the ideal of all compact operators from $X$ to $Y$ is represented by the set $\mathcal{K}(X,Y)$. $\mathcal{B}(X)$ and $\mathcal{K}(X)$ represent $\mathcal{B}(X,Y)$ and $\mathcal{K}(X,Y)$, respectively, if $X = Y$. The operator norm for any $T \in \mathcal{B}(X)$ is $\|T\|$, and the supremum norm on the sequence space is $\|.\|_{\infty}$. Let $T: X \rightarrow Y$ be a bounded linear operator. Here, $R(T)$ and $N(T)$ represent the range and null spaces of $T$, respectively. The adjoint operator $T^*: Y^* \rightarrow X^*$ of $T$, is a bounded linear operator, defined as
		\[
		(T^* f)(x) = f(Tx) \quad \text{for all } f \in Y^* \text{ and } x \in X,
		\]
		where $X^*$ and $Y^*$ are the dual spaces of $X$ and $Y$ respectively. 
		
		Consider a complex Banach space $X \neq \{0\}$, and $T \in \mathcal{B}(X)$. The set of all complex integers $\lambda$ for which the operator $T - \lambda I$ is invertible is known as the resolvent set of $T$, where $I$ is the identity operator on $X$, and it is denoted by $\rho(T, X)$. Thus,
		
		\begin{equation*}
			\rho(T, X)= \{\lambda \in \mathbb{C} : N(T - \lambda I) = \{0\} \text{ and } R(T - \lambda I) = X\}. 
		\end{equation*}
		
		The complement of the resolvent set in the complex plane \(\mathbb{C}\) is called the spectrum of $T$, and it is denoted by \(\sigma(T, X)\). Therefore,
		
		\begin{equation*}
			\sigma(T, X)  
			=\{\lambda \in \mathbb{C} : N(T - \lambda I) \neq \{0\} \text{ or } R(T - \lambda I) \neq X\}.
		\end{equation*}

		The spectrum of an operator \(T\) can be divided into three classes based on the sets \(N(T - \lambda I)\) and \(R(T - \lambda I)\). If a point \(\lambda \in \mathbb{C}\) is in the spectrum, then there are two cases: (i) \(N(T - \lambda I) \neq \{0\}\), (ii) \(N(T - \lambda I) = \{0\}\) but \(R(T - \lambda I) \neq X\).
		For the case where \(N(T - \lambda I) = \{0\}\) but \(R(T - \lambda I) \neq X\), there are further two possible outcomes for $R(T - \lambda I)$ which are \(\overline{R(T - \lambda I)} \neq X\) or \( \overline{R(T - \lambda I)} = X\). Based on this, we can define the following three sets:
		
		\begin{itemize}
			\item [(i)]The \textit{point spectrum} of \(T\) is the set of all \(\lambda \in \mathbb{C}\) for which \(N(T - \lambda I) \neq \{0\}\), denoted by \(\sigma_p(T, X)\), 
			\item [(ii)] 
			The \textit{continuous spectrum} of \(T\) is the set of all \(\lambda \in \mathbb{C}\) for which \(N(T - \lambda I) = \{0\}\) and $\overline{(R(T - \lambda I)} = X$ but $R(T - \lambda I) \ne X$, denoted by \(\sigma_p(T, X)\),
			
			\item [(iii)]
			The \textit{residual spectrum} of \(T\) is the set of all \(\lambda \in \mathbb{C}\) for which \(N(T - \lambda I) = \{0\}\) and \( \overline{R(T - \lambda I)} \neq X\), denoted by \(\sigma_p(T, X)\).
		\end{itemize}

		The sets \(\sigma_p(T, X), \sigma_r(T, X), \sigma_c(T, X)\) are disjoint and forms a partition of \(\sigma(T, X)\). The determination of point spectrum, continuous spectrum and residual spectrum of an operator is
		called the fine spectra.
		For $T \in \mathcal{B}(X)$, we call a sequence $\{x_k\} \in X$ a Weyl sequence for $T$ if $\left\|x_k\right\|=1$ and $\left\|T x_k\right\| \rightarrow 0$ as $k \rightarrow \infty$.
		The approximate point spectrum of $T$ is defined as
		$$\sigma_{a p}(T):=\{\lambda \in \mathbb{K} : \text{there exists a Weyl sequence for}~ T-\lambda I\}.$$
		Moreover, defect spectrum $\sigma_\delta(T)$ and compression spectrum $\sigma_{c o}(T)$ of $T$ are defined as 
		$$
		\sigma_\delta(T):=\{\lambda \in \mathbb{K}: T-\lambda I \text { is not surjective }\}
		$$
		$$
		\sigma_{c o}(T):=\{\lambda \in \mathbb{K}: \overline{R(T-\lambda I)} \neq X\}.
		$$
		There are some other subdivisions of $\sigma(T)$ (not necessarily disjoint) as
		\begin{eqnarray*}
			\sigma(T)=\sigma_{a p}(T) \cup \sigma_\delta(T)  \\
			\sigma(T)=\sigma_{a p}(T) \cup \sigma_{c o}(T).
		\end{eqnarray*}
		
		The space of all complex sequences is represented by $\mathbb{C}^\mathbb{N}$, and $r = \{r_k\}$ is an infinite positive real sequence. With the norm $\|x\|_r=\sup_k |x_k|r_k$, the weighted null sequence space $c_0(r)$ is defined as follows: $c_0(r)=\{ \{x_k\}\in \mathbb{C}^{\mathbb{N}}: \lim_{k\to\infty} r_kx_k=0\}$. Under this norm, $(c_0(r),\|x\|_r)$ is a Banach space. The dual of $c_0(r)$ is linearly isometric to the weighted sequence space $\ell_1(r^{-1})$, where $r^{-1} = \{\frac{1}{r_n}\}.$ Furthermore, it should be noted that $c_0(r) = c_0$ and the norms are the same if $\inf_k r_k > 0$. We are therefore interested in the situation where $\inf_k r_k = 0.$.

		Some inclusion relations of the spectrum of a bounded linear operator and its adjoint operator are mentioned in the subsequent proposition..
	
			From  Goldberg's classifications \cite{goldberg_unbounded}, if  $X$ is a Banach space and $T \in \mathcal{B}(X)$, then there are three possibilities for each $R(T-\lambda I)$ and $(T-\lambda I)^{-1}$ which are
			\begin{enumerate}
				\item [(I)] $R(T-\lambda I)=X$,
				\item [(II)] $R(T-\lambda I)\ne \overline{R(T-\lambda I)}=X$,
				\item[(III)] $\overline{R(T-\lambda I)}\ne X$,
			\end{enumerate} 
			and
			\begin{enumerate}
				\item[(1)] $(T-\lambda I)^{-1}$ exists and its continuous,
				\item[(2)]  $(T-\lambda I)^{-1}$ exists and its not continuous,
				\item[(3)] $(T-\lambda I)^{-1}$ does not exist.
			\end{enumerate}  
			If these possibilities are combined, then nine different states are created. These are labeled by $I_1$, $I_2$, $I_3$, $II_1$, $II_2$, $II_3$, $III_1$, $III_2$, $III_3$. For example, if $T\in II_2$ then $R(T-\lambda I)\ne \overline{R(T-\lambda I)}=X$ and  $(T-\lambda I)^{-1}$ exists and its not bounded.

			The following lemma is useful in this sequel.
			\begin{lemma} \cite{wilansky_sequence} \label{bounded_c0}
				The matrix $A = (a_{nk})(n,k=1,2,3,...)$ gives rise to a bounded linear operator $T \in \mathcal{B}(c_0)$ if and only if the following conditions hold,
				\begin{enumerate}
					\item[(i)] the rows of $A$ are in $\ell_1$ and their $\ell_1$ norms are bounded,
					\item[(ii)] the columns of $A$ are in $c_0$.
				\end{enumerate}
				The operator norm of $T$ is given by the supremum of the $\ell_1$ norms of the rows.
			\end{lemma}
			
			\section{Spectrum of Compact Rhaly Operators}
			In the following theorem, a sufficient condition for the boundedness and compactness of $ R_a : c_0(s) \to c_0(s)$ is given when the weight vectors for both the domain and range spaces are the same.
			
			\begin{theorem}\label{cor11}
				If $a=\{a_n\}$ is a sequence of real or complex numbers and $s=\{s_n\}$ is a sequence of strictly positive real numbers such that $0 \leq \limsup_{n \rightarrow \infty} \frac{s_{n+1}}{s_n}<1$ then $R_a$ is a bounded linear operator on $c_0(s)$ if $\{a_n\}$ is a bounded sequence and $R_a$ is a compact operator on $c_0(s)$ if $\{a_n\}$ is a null sequence.
			\end{theorem}
			
			\begin{proof}
				Since $0 \leq \limsup_{n \rightarrow \infty} \frac{s_{n+1}}{s_n}<1$, there exists $0<\eta<1$ and $n_0 \in \mathbb{N}$ such that
				\[s_{n+1} \le \eta s_n \ \ \ \forall \ \ n \geq n_0.\]
				Using the aforementioned relation and for fixed $n > n_0$, we have
				\begin{eqnarray*}
					&&s_{n} \leq \eta^{n-k} s_k \ \ \ \forall \ \ n_0 \leq k \leq n,\\
					&& \frac{1}{s_k} \leq \frac{\eta^{n-k}}{s_n} \ \ \ \forall \ \ n_0 \leq k \leq n.
				\end{eqnarray*}
				Thus we have,
				\begin{equation} \label{4.2.2a}
					\sum\limits_{k=n_0}^{n}\frac{1}{s_k} \leq \sum\limits_{k=n_0}^{n}\frac{\eta^{n-k}}{s_n}\leq \frac{1}{(1-\eta)s_n}.
				\end{equation}
				Take $M_0 = \sum\limits_{k=1}^{n_0-1}\frac{1}{s_k}.$ Then
				\begin{equation*}
					\sum\limits_{k=1}^{n}\frac{1}{s_k} = M_0 + \sum\limits_{k=n_0}^{n}\frac{1}{s_k}.
				\end{equation*}
				From \eqref{4.2.2a}, we obtained that
				\begin{eqnarray*}
					s_n |a_n| \sum\limits_{k=1}^{n}\frac{1}{s_k} &\leq& |a_n| \left( M_0 \norm{s}_{\infty} +  \frac{1}{(1-\eta)}\right).
				\end{eqnarray*}
				The results follow from the above relation, Theorem 3.1 \cite{patra2023spectral} and Theorem 3.2 \cite{patra2023spectral}.
			\end{proof}
			
			Theorem \ref{cor11} provides good source of weights $\{s_n\}.$ The following example implies that Theorem \ref{cor11} is merely sufficient condition.
			
			\begin{example}	
				Let $a_n = \frac{1}{n^p}, \ p>0.$ Then $R_a$ becomes the $p$-Ces\`aro operator defined by Rhaly \cite{rhaly} and for $p=1$ it becomes the Ces\`aro operator $C_1.$ It is straightforward to demonstrate that for  $0<p< 1$ the $\ell_1$ norms of the rows of $R_a$ are not bounded and Lemma \ref{bounded_c0} implies that $R_a \notin \mathcal{B}(c_0).$ Take $s_n = e^{-n^{\frac{1}{\alpha}}}$ for $n \in \mathbb{N}$ where $p > 1$ and $\alpha \in \mathbb{N}\setminus \{1\}.$ Now
				%Since
				%\[n^{\frac{1}{\alpha}} - (n+1)^{\frac{1}{\alpha}} \simeq \frac{1}{\alpha n^{1-\frac{1}{\alpha}}},\]
				%then
				\[\lim\limits_{n \rightarrow \infty} \frac{s_{n+1}}{s_n} =\lim\limits_{n \rightarrow \infty} \frac{e^{-(n+1)^{\frac{1}{\alpha}}}}{e^{-n^{\frac{1}{\alpha}}}} = \lim\limits_{n \rightarrow \infty} e^{\left( n^{\frac{1}{\alpha}} - (n+1)^{\frac{1}{\alpha}}\right)} =  1.\]
				Hence, Theorem \ref{cor11} fails to give any conclusion. However
				\[\sum\limits_{k=1}^{n} \frac{1}{s_k} =\sum\limits_{k=1}^{n} e^{k^{\frac{1}{\alpha}}} \leq \int\limits_{1}^{n+1}  e^{x^{\frac{1}{\alpha}}} dx = \int\limits_{1}^{(n+1)^{\frac{1}{\alpha}}}  t^{\alpha - 1} e^t dt,\]
				where
				\begin{eqnarray*}
					\int\limits_{1}^{(n+1)^{\frac{1}{\alpha}}}  t^{\alpha - 1} e^t dt &=& \left[ e^t t^{\alpha -1} +  e^t \sum\limits_{j=1}^{\alpha - 1} (-1)^{j} t^{\alpha-(j+1)} \prod\limits_{l=1}^{j}(\alpha - l) \right]_1^{(n+1)^\frac{1}{\alpha}}\\
					&=& e^{(n+1)^{\frac{1}{\alpha}}} \left[ (n+1)^\frac{\alpha -1}{\alpha} + \mathcal{O}((n+1)^\frac{\alpha -2}{\alpha}) \right],
				\end{eqnarray*}
				and hence, that
				\[a_ns_n\sum\limits_{k=1}^{n} \frac{1}{s_k} \leq \frac{e^{(n+1)^{\frac{1}{\alpha}}} \left[ (n+1)^\frac{\alpha -1}{\alpha} + \mathcal{O}((n+1)^\frac{\alpha -2}{\alpha}) \right]}{e^{n^{\frac{1}{\alpha}}} n^p}.\]
				
				The above relation implies that $\lim_{n \rightarrow \infty} a_n s_n \sum\limits_{k=1}^{n} \frac{1}{s_k} = 0$, by Theorem 3.2 \cite{patra2023spectral}, we have $R_a \in \mathcal{K}(c_0(s))$ . Consequently, it follows that $R_a \in \mathcal{B}(c_0(s))$.
				
			\end{example}

			Next, we analyze the spectrum and fine spectra of the compact Rhaly operator $R_a: c_0(s) \to c_0(s)$.
			From now on, we assume that $a = \{a_n\}$ is a sequence of non-zero real numbers and $\{s_n\}$ is a bounded decreasing sequence of strictly positive real numbers.
			Consider the set $S = \{a_n: n \in \mathbb{N}\}.$ The following theorem describes the spectrum and fine spectrum of the compact Rhaly operator $R_a$ in the weighted sequence space $c_0(s)$.
			\begin{theorem}\label{th_spectra}
				The compact Rhaly operator $R_a:c_0(s) \to c_0(s)$ satisfies the following relations.
				\begin{enumerate}
					\item[(i)] $\sigma_p(R_a, c_0(s)) = S,$
					\item[(ii)] $\sigma(R_a, c_0(s)) = S \cup \{0\},$
					\item[(iii)] $\sigma_c(R_a, c_0(s)) = \{0\},$
					\item[(iv)] $\sigma_r(R_a, c_0(s))=\phi$. 
				\end{enumerate}
			\end{theorem}
			
			\begin{proof}
				\allowdisplaybreaks
				\begin{enumerate}
					\item[(i)] Let us consider $R_ax=\lambda x$ for some $\lambda \in \mathbb{C}$ and non-zero $x\in c_0(s).$ We have
					\begin{equation} \label{eqn2}
						a_k\sum_{j=1}^kx_j = \lambda x_k ~ \text{for all}~k\in \mathbb{N}.
					\end{equation}
					If $\lambda \notin S$, then above equations imply that $x_k = 0$ for all $k \in \mathbb{N}$. Therefore, $\sigma_p(R_a, c_0(s)) \subseteq S.$

					\noindent Let $\lambda \in S$ and $m$ be the smallest integer such that $x_m \neq 0.$ Then, by solving the system of equations (\ref{eqn2}), we have $\lambda = a_m$ and
					\[x_n = \left( \prod\limits_{j=m+1}^{n} \frac{\lambda a_{j-1}^{-1}}{\lambda a_j^{-1} - 1} \right)x_m, \ n = m+1, m+2, \cdots.\]
					Now it is required to prove that $\{x_n\} \in c_0(s)$. Take $q_n = \frac{s_{n+1}}{s_n^2|a_n|^2\sum_{k=1}^{n}\frac{1}{s_k}}, \ n \in \mathbb{N}$. Now,
					\begin{eqnarray}\label{4.2.4}
						&&q_n \frac{\vert{x_ns_n}\vert^2}{\vert{x_{n+1}s_{n+1}}\vert^2} - q_{n+1} \nonumber \\
						&=& \frac{1}{\lambda^2 s_{n+1} a_{n+1}^2 \sum_{k=1}^{n}\frac{1}{s_k}}(\lambda - a_{n+1})^2 - \frac{1}{s_{n+1}^2a_{n+1}^2\sum_{k=1}^{n+1}\frac{1}{s_k}}s_{n+2} \nonumber \\
						% &=& \frac{1}{\lambda^2 s_{n+1}^2 a_{n+1}^2 \sum_{k=1}^{n+1}\frac{1}{s_k}\sum_{k=1}^{n}\frac{1}{s_k}}\left(s_{n+1}\sum_{k=1}^{n+1}\frac{1}{s_k} (\lambda - a_{n+1})^2 - \lambda^2 s_{n+2} \sum_{k=1}^{n}\frac{1}{s_k}\right)\nonumber  \\
						% &=& \frac{1}{\lambda^2 s_{n+1}^2 a_{n+1}^2 \sum_{k=1}^{n+1}\frac{1}{s_k}\sum_{k=1}^{n}\frac{1}{s_k}}\left(s_{n+1}\sum_{k=1}^{n+1}\frac{1}{s_k} (\lambda^2 +a_{n+1}^2-2\lambda a_{n+1}) - \lambda^2 s_{n+2} \sum_{k=1}^{n}\frac{1}{s_k}\right)\nonumber  \\
						&=&\frac{1}{\lambda^2 \left(s_{n+1}a_{n+1}\sum_{k=1}^{n+1}\frac{1}{s_k} \right)\left( s_{n+1}a_{n+1}\sum_{k=1}^{n}\frac{1}{s_k}  \right)} \lambda^2\left(s_{n+1}\sum_{k=1}^{n+1}\frac{1}{s_k}-s_{n+2}\sum_{k=1}^{n}\frac{1}{s_k}\right) \nonumber\\
						&- & \frac{1}{\lambda^2 \left(s_{n+1}a_{n+1}\sum_{k=1}^{n+1}\frac{1}{s_k} \right)\left( s_{n+1}a_{n+1}\sum_{k=1}^{n}\frac{1}{s_k}  \right)}\left(2 \lambda a_{n+1}s_{n+1}\sum_{k=1}^{n+1}\frac{1}{s_k}-a_{n+1}^2s_{n+1}\sum_{k=1}^{n+1}\frac{1}{s_k}\right).\nonumber\\
						&&
					\end{eqnarray}
					Since $\{s_n\}$ is decreasing sequence, we obtain
					\begin{equation} \label{4.2.5}
						s_{n+1}\sum_{k=1}^{n+1}\frac{1}{s_k}-s_{n+2}\sum_{k=1}^{n}\frac{1}{s_k} = 1+(s_{n+1} - s_{n+2})\sum_{k=1}^{n}\frac{1}{s_k} \geq 1, \forall n \in \mathbb{N}.
					\end{equation}
					Using equations \eqref{4.2.5} and \eqref{4.2.4}, we obtain that
					\begin{align}\label{ch4_kummer}
						&q_n \frac{\vert{x_ns_n}\vert^2}{\vert{x_{n+1}s_{n+1}}\vert^2} - q_{n+1} \nonumber\\
						\geq& \frac{\lambda^2}{\lambda^2 \left(s_{n+1}a_{n+1}\sum_{k=1}^{n+1}\frac{1}{s_k} \right)\left( s_{n+1}a_{n+1}\sum_{k=1}^{n+1}\frac{1}{s_k}-a_{n+1}  \right)} \nonumber\\ 
						- & \frac{1}{\lambda^2 \left(s_{n+1}a_{n+1}\sum_{k=1}^{n+1}\frac{1}{s_k} \right)\left( s_{n+1}a_{n+1}\sum_{k=1}^{n+1}\frac{1}{s_k}-a_{n+1}  \right)}\left(2 \lambda a_{n+1}s_{n+1}\sum_{k=1}^{n+1}\frac{1}{s_k}-a_{n+1}^2s_{n+1}\sum_{k=1}^{n+1}\frac{1}{s_k}\right).
					\end{align}
					Since $R_a$ is a compact operator over $c_0(s)$, from Theorem 3.2 \cite{patra2023spectral} we have $\left\lbrace a_{n}s_{n}\sum_{k=1}^{n}\frac{1}{s_k} \right\rbrace \in c_0.$ Then from the relation \eqref{ch4_kummer} it follows that
					\[\lim \left(q_n \frac{\vert{x_ns_n}\vert^2}{\vert{x_{n+1}s_{n+1}}\vert^2} - q_{n+1}\right) \to +\infty~\text{as}~n \rightarrow \infty.\]
					According to Kummer's Test \cite[p. 395]{kummer}, $\{x_ns_n\} \in \ell_2$, which implies that $\{x_n\} \in c_0(s).$ Therefore, $S \subseteq \sigma_p(R_a, c_0(s)).$

					\item[(ii)] 
					As $\sigma_p(R_a, c_0(s))\subseteq\sigma(R_a, c_0(s)) $ and $\sigma(R_a, c_0(s))$ is closed, we have
					\begin{equation*}
						S\cup \{0\} \subseteq\sigma(R_a, c_0(s)).
					\end{equation*}
					Since $R_a \in \mathcal{K}(c_0(s))$, then from Theorem 3.4.23\cite{megginson}, we have $\sigma(R_a, c_0(s)) \subseteq S \cup \{0\}.$ Hence, $\sigma(R_a, c_0(s)) = S \cup \{0\}.$
					\item[(iii)]
					Since $0 \notin \sigma_p(R_a, c_0(s))$,
					$R_a^{-1}$ exists.
					Suppose $R_a^*x=0$ where $x \in \ell_1(s^{-1})$. Then we have
					\begin{equation*}
						\begin{aligned}
							a_1x_1 + a_2x_2 + a_3x_3 + a_4x_4 + \cdots &=&  0\\
							a_2x_2 + a_3x_3 + a_4x_4 + \cdots &=& 0\\
							a_3x_3 + a_4x_4 + \cdots &=&  0\\
							&\vdots&.
						\end{aligned}
					\end{equation*}
					Through a straightforward calculation, we obtain that
					$$a_1x_1=0,a_2x_2=0, \cdots a_nx_n=0 \cdots.$$
					As $a_i \ne 0$ for all $i$, we have $x_i=0$ for all $i$. This implies $R_a^*$ is one-one. Thus, Theorem II.3.7 \cite[p. 59]{goldberg_unbounded} gives us, $\overline{R(R_a)}=c_0(s)$. Also $R(R_a)\ne c_0(s)$ as $0 \in \sigma(R_a, c_0(s))$.
			
					If $R_a^{-1}$ is bounded then by Theorem 3.2.11 \cite{nair2021functional}, we have $R_a$ is bounded below. Since $c_0(s)$ is a Banach space, then by Theorem 3.2.12\cite{nair2021functional}, we have $R(R_a)$ is closed. This implies $R(R_a)=c_0(s)$. Thus $0 \in \rho(R_a)$ which is a contradiction. Thus, $R_a^{-1}$ is unbounded. 
					Hence, $0 \in \sigma_c(R_a, c_0(s))$. 
					Since $\sigma_p(R_a, c_0(s)), \sigma_r(R_a, c_0(s))$ and $\sigma_c(R_a, c_0(s))$ are disjoint and forms a partition of $\sigma(R_a, c_0(s))$. Hence, $\sigma_c(R_a, c_0(s)) = \{0\}.$
					% Thus, the required result holds.
					
					\item[(iv)] Since $\sigma_p(R_a, c_0(s)), \sigma_r(R_a, c_0(s))$ and $\sigma_c(R_a, c_0(s))$ forms a partition of $\sigma(R_a, c_0(s))$, we have $\sigma_r(R_a, c_0(s))=\phi$.
				\end{enumerate}
			\end{proof}
			
			From the table of Goldberg's classification \cite[p.196]{rev1}, we have the following relations
			\begin{align} \label{pr}
				\sigma_p\left(R_a, c_0(s)\right)&=I_3\sigma\left(R_a, c_0(s)\right) \cup II_3\sigma\left(R_a, c_0(s)\right) \cup III_3\sigma\left(R_a, c_0(s)\right),\\
				\sigma_r\left(R_a, c_0(s)\right)&=III_1\sigma\left(R_a, c_0(s)\right) \cup III_2\sigma\left(R_a, c_0(s)\right), \label{rr}\\
				\sigma_c\left(R_a, c_0(s)\right)&=II_2\sigma\left(R_a, c_0(s)\right) \label{cr}.
			\end{align}
			Based on these relations, we can simply obtain the following result..
			\begin{theorem}\label{Gold2}
				The operator $R_a \in\mathcal{K}(c_0(s))$ satisfies the following relations.
				\begin{itemize}
					\item [(i)] $I_3 \sigma \left(R_a, c_0(s)\right)=II_3 \sigma \left(R_a, c_0(s)\right)=\phi$,
					\item [(ii)] $III_3 \sigma \left(R_a, c_0(s)\right)=S$,
					\item [(iii)] $II_2 \sigma \left(R_a, c_0(s)\right)=\{0\}$,
					\item [(iv)] $III_1 \sigma \left(R_a, c_0(s)\right)=III_2 \sigma \left(R_a, c_0(s)\right)=\phi$.
				\end{itemize}
			\end{theorem}
			\begin{proof}
				Consider the equation $R_a^* x = \lambda x$ for some $\lambda \in \mathbb{C}.$ This gives the following equations
				\begin{equation*}
					\begin{aligned}
						a_1x_1 + a_2x_2 + a_3x_3 + a_4x_4 + \cdots &=& \lambda x_1\\
						a_2x_2 + a_3x_3 + a_4x_4 + \cdots &=& \lambda x_2\\
						a_3x_3 + a_4x_4 + \cdots &=& \lambda x_3\\
						&\vdots&
					\end{aligned}
				\end{equation*}
				From the above equations it follows that for all $n \in \mathbb{N}$ and $n \geq 2$
				\begin{equation*}
					x_n = \prod\limits_{j=1}^{n-1} \left(1- \frac{a_j}{\lambda} \right) x_1.
				\end{equation*}
				If $\lambda = a_k$ for some $k \in \mathbb{N}$ then from the above equation we get
				\[x_{k+1} = x_{k+2} = \cdots = 0.\]
				Hence $\{x_n\} \in c_0(s)^*.$ It follows that every $a_k,$ $k \in \mathbb{N}$ is an eigenvalue of $R_a^*$ i.e. $S \subseteq \sigma_p(R_a^*,c_0^*(s)$. 
				Thus, for \(\lambda \in S\), the operator \((R_a - \lambda I)^*\) is not injective. By Theorem II.3.7 \cite[p. 59]{goldberg_unbounded}, it follows that \(\overline{R(R_a - \lambda I)} \neq c_0(s)\) for all \(\lambda \in S\). Consequently, the results in (i) and (ii) follow directly from the relation given in \(\eqref{pr}\).
				Since $\sigma_c\left(R_a, c_0(s)\right)=II_2\sigma\left(R_a, c_0(s)\right)$, (iii) follows from Theorem \ref{th_spectra}(iii). The result in (iv) follows from the relation (\ref{cr}) and Theorem \ref{th_spectra}(iv).
			\end{proof}
			Moreover, 
			From the table of Goldberg's classification \cite[p.196]{rev1}, we have the following relations
			\begin{equation}\label{del}
				\begin{aligned}
					\sigma_\delta\left(R_a, c_0(s)\right) &= II_2\sigma\left(R_a, c_0(s)\right) \cup II_3\sigma\left(R_a, c_0(s)\right) \cup III_1\sigma\left(R_a, c_0(s)\right) \\
					&\quad \cup III_2\sigma\left(R_a, c_0(s)\right) \cup III_3\sigma\left(R_a, c_0(s)\right),
				\end{aligned}
			\end{equation}
			
			and
			\begin{equation} \label{co}
				\sigma_{co}\left(R_a, c_0(s)\right)=III_1\sigma\left(R_a, c_0(s)\right) \cup III_2\sigma\left(R_a, c_0(s)\right)\cup III_3\sigma\left(R_a, c_0(s)\right).
			\end{equation} 
			\begin{theorem}\label{adc}
				The approximate point spectrum, defect spectrum, and compression spectrum of $R_a \in\mathcal{K}(c_0(s))$ are given by
				\begin{itemize}
					\item [(i)] $\sigma_{a p}\left(R_a, c_0(s)\right)=S \cup \{0\} $,
					\item [(ii)] $\sigma_\delta\left(R_a, c_0(s)\right)=S \cup \{0\} $,
					\item [(iii)] $\sigma_{c o}\left(R_a, c_0(s)\right)=S$.
				\end{itemize}
			\end{theorem} 
			\begin{proof}
				Since, $\sigma_{a p}\left(R_a, c_0(s)\right) \subseteq \sigma\left(R_a, c_0(s)\right) $, thus $\sigma_{a p}\left(R_a, c_0(s)\right) \subseteq S \cup \{0\} $. Moreover, $\sigma_{p}\left(R_a, c_0(s)\right) \subseteq \sigma_{a p}\left(R_a, c_0(s)\right)$ and $\sigma_{a p}\left(R_a, c_0(s)\right)$ is closed set, we have $S \cup \{0\} \subseteq \sigma_{a p}\left(R_a, c_0(s)\right). $ Thus, (i) holds. Using Theorem \ref{adc}, the results in (ii) and (iii) hold from relations (\ref{del}) and (\ref{co}) respectively. 
			\end{proof}
			\section{Goldberg's classifications of Rhaly Operator $R_a$ over $c_0(s)$.}
			In this section, let $a=\{a_n\}$ is a sequence of positive real numbers and $S = \{a_n: n \in \mathbb{N}\}$. In \cite{patra2023spectral}, the authors have obtained the spectrum of the Rhaly operator $R_a$ over $c_0(s)$ with certain conditions on the weight vector. Using this information, in this section, we will derive Goldberg's classifications of the Rhaly operator $R_a$ over $c_0(s)$. The following result on the spectrum of the Rhaly operator is useful.

			Consider,
			
			$
			\begin{aligned}
				& A_1=\left\{\lambda \in S: \lim _{n \rightarrow \infty} a_n s_n n^{\alpha \chi}=0\right\}, \\
				& A_2=\left\{\lambda \in \mathbb{C} \backslash(S \cup\{0\}): \sum_{n=1}^{\infty} \frac{1}{s_n n^{\alpha \chi}}<\infty\right\},
			\end{aligned}
			$
			where $\alpha=\Re\left(\frac{1}{\lambda}\right)$.

			\begin{theorem}\label{gold1}
				Let $\left\{s_n\right\}$ be a bounded and strictly positive sequence such that $R_a \in$ $\mathcal{B}\left(c_0(s)\right)$ and $\lim _{n \rightarrow \infty} n a_n=\chi \neq 0$ then the following statements hold:
				\begin{itemize}
					\item [(i)] $I_3 \sigma \left(R_a, c_0(s)\right)= II_3 \sigma \left(R_a, c_0(s)\right)=\phi$,
					\item [(ii)] $III_3 \sigma \left(R_a, c_0(s)\right)=A_1$,
					\item [(iii)] $A_2 \subseteq III_1 \sigma \left(R_a, c_0(s)\right)$ and $III_2 \sigma \left(R_a, c_0(s)\right) \subseteq S\setminus A_1 $.
				\end{itemize}
				In addition, if the sequence $\left\{s_n\right\}$ is a decreasing sequence then
				\begin{itemize}
					\item [(iv)] $ 0 \in II_2 \sigma \left(R_a, c_0(s)\right)$.
				\end{itemize}
			\end{theorem}
			\begin{proof}
				The results in (i) and (ii) follow from Theorem 4.9(i) \cite{patra2023spectral}, Theorem 4.9(ii) \cite{patra2023spectral}, relation (\ref{pr}), and Theorem II.3.7 \cite[p. 59]{goldberg_unbounded}. From the proof of Theorem 4.8 in \cite{patra2023spectral}, it follows that for any $\lambda \notin \bar{S}$, $(R_a - \lambda I)^{-1}$ exists and bounded. This implies $A_2 \subseteq III_1 \sigma \left(R_a, c_0(s)\right)$. Moreover, $A_2$ and $S\setminus A_1$ are disjoint,  consequently, Theorem 4.9(iii) \cite{patra2023spectral} and relation (\ref{rr}) give us, $\sigma_r(R_a, c_0(s))\setminus(S \setminus A_1) \subseteq \sigma_r(R_a, c_0(s))\setminus III_2\sigma(R_a, c_0(s))$. Hence, $III_2 \sigma \left(R_a, c_0(s)\right) \subseteq S \setminus A_1$. The result in (iv) follows from Theorem 4.9(v) \cite{patra2023spectral} and relation (\ref{cr}).
			\end{proof}
			
			\begin{theorem}
				The approximate point spectrum, defect spectrum, and compression spectrum of $R_a$ over $c_0(s)$ under the same conditions as in Theorem \ref{gold1} satisfy the subsequent relations,
				
				\begin{itemize}
					\item [(i)] $ \sigma\left(R_a, c_0(s)\right)\setminus (A_2 \cup S) \subseteq \sigma_{a p}\left(R_a, c_0(s)\right)\subseteq \sigma\left(R_a, c_0(s)\right) \setminus A_2  $,
					\item [(ii)] $\sigma\left(R_a, c_0(s)\right)\setminus A_1\subseteq \sigma_\delta\left(R_a, c_0(s)\right)$,
					\item [(iii)] $\sigma_{c o}\left(R_a, c_0(s)\right)=A_2 \cup S$.
				\end{itemize}
			\end{theorem} 
			\begin{proof}
				\begin{itemize}
					\item [(i)] 
					The relation $\sigma_{p}\left(R_a, c_0(s)\right) \subseteq \sigma_{ap}\left(R_a, c_0(s)\right)$ and Proposition 5.2.1(g) \cite[p.195]{basar2012summability}, follow that $\sigma\left(R_a, c_0(s)\right)\setminus \sigma_p\left(R_a^*, c_0(s)^*\right) \subseteq  \sigma_{ap}\left(R_a, c_0(s)\right)$. This implies $$\sigma\left(R_a, c_0(s)\right)\setminus (A_2 \cup S) \subseteq \sigma_{ap}\left(R_a, c_0(s)\right).$$ From Table 1, we have $\sigma_{ap}\left(R_a, c_0(s)\right)=\sigma\left(R_a, c_0(s)\right)\setminus III_1$. Using Theorem \ref{gold1}(iii), we have $\sigma_{ap}\left(R_a, c_0(s)\right) \subseteq \sigma\left(R_a, c_0(s)\right) \setminus A_2 $.  
					\item [(ii)] From Proposition 5.2.1(g) \cite[p.195]{basar2012summability}, we have $\sigma\left(R_a, c_0(s)\right)\setminus \sigma_p\left(R_a, c_0(s)\right)\subseteq\sigma_{ap}\left(R_a^*, c_0(s)^*\right)$. Thus, by using Proposition 5.2.1(c) \cite[p.195]{basar2012summability} and Theorem \ref{gold1}(ii), we can obtain the required result easily.
					\item [(iii)] The result in (iii) follows from Proposition 5.2.1(g) \cite[p.195]{basar2012summability}.
				\end{itemize}
			\end{proof}
			
			\section{Spectrum of Compact Generalised Ces{\`a}ro Operator}	
			In this section, the boundedness and compactness criteria of the generalized Cesàro operator \( C_t \) are investigated. Additionally, its spectral subdivisions are analyzed. Here, \( r = \{r_n\} \) and \( s = \{s_n\} \) are considered as bounded strictly positive weight vectors.
			\begin{theorem}\label{bt}
				The Generalised Ces{\`a}ro operator $C_t : c_0(r) \to c_0(s)$ is bounded if and only if
				\[\left\lbrace \frac{s_n}{n}  \sum\limits_{k=1}^{n}\frac{t^{n-k}}{r_k}\right\rbrace_{n \in \mathbb{N}} \in \ell_{\infty}, \]
				and in this case, $\norm{R_a}_{r,s} = \sup\limits_n \frac{s_n}{n+1}  \sum\limits_{k=1}^{n}\frac{t^{n-k}}{r_k}.$
			\end{theorem}	
			\begin{proof}
				Since,  \[ C_t = (a_{nk})_{n,k=1}^\infty = \left\lbrace \begin{aligned}
					\frac{t^{n-k}}{n}, & \quad n \geq k \\
					0, & \quad n < k 
				\end{aligned} \right. \] and $\{s_n\}$ is a bounded sequence, thus $\lim_{ n\to \infty}s_k\frac{t^{n-k}}{n}=0$. Then, the required result is a direct consequence of Theorem 1 \cite{icmc}.
			\end{proof}
			
			\begin{theorem}\label{ct}
				The Generalised Ces{\`a}ro operator $C_t : c_0(r) \to c_0(s)$ is compact if and only if
				\[\left\lbrace \frac{s_n}{n}  \sum\limits_{k=1}^{n}\frac{t^{n-k}}{r_k}\right\rbrace_{n \in \mathbb{N}} \in c_{0}. \]
			\end{theorem}
			\begin{proof}
				Let us consider the linear operator $T_{r,s}: \mathbb{C}^{\mathbb{N}} \rightarrow  \mathbb{C}^{\mathbb{N}}$ defined as
				\[T_{r,s}(x) = \left\lbrace \frac{s_n}{n} \sum\limits_{k=1}^{n} \frac{t^{n-k}x_k}{r_k} \right\rbrace_{n \in \mathbb{N}}.\]
				Then $D_s C_t = T_{r,s}D_r$ where $D_r$ and $D_s$ are the diagonal matrix with $i$-th diagonal entry $r_i$ and $s_i$ respectively. Consequently $C_t : c_0(r) \to c_0(s)$ is a compact map if and only if $T_{r,s} \in \mathcal{K}( c_0).$
				Let $\left\lbrace \frac{s_n}{n} \sum\limits_{k=1}^{n}\frac{t^{n-k}}{r_k}\right\rbrace_{n \in \mathbb{N}} \in c_0$. Define a sequence of operators $\{T_{r,s}^{(k)}\}_{k \in \mathbb{N}}$ where $T_{r,s}^{(k)}:c_0 \rightarrow c_0 $ and
				\[T_{r,s}^{(k)}(x)=\left\lbrace s_1\frac{x_1}{r_1}, \frac{s_2}{2}\left(\frac{tx_1}{r_1}+\frac{x_2}{r_2}\right), \cdots, \frac{s_k}{k} \sum\limits_{j=1}^{k} \frac{t^{k-j}x_j}{r_j},0,0, \cdots \right\rbrace.\]
				The operator $T_{r,s}^{(k)}$ is a finite rank operator for each $k\in \mathbb{N}$. Now
				\[
				\norm{(T_{r,s} - T_{r,s}^{(k)})x}_\infty = \sup\limits_{n>k} \frac{s_n}{n} \sum\limits_{j=1}^{n} \frac{t^{n-j}x_j}{r_j} \leq \norm{x}_\infty  \sup\limits_{n>k} \frac{s_n}{n} \sum\limits_{j=1}^{n} \frac{t^{n-j}}{r_j}.
				\]
				Hence,
				\[\norm{T_{r,s}- T_{r,s}^{(k)}} \leq \sup\limits_{n>k} \frac{s_n}{n} \sum\limits_{j=1}^{n}\frac{t^{n-j}}{r_j}.\]
				By taking $k \rightarrow \infty$ on both side of the above relation we get $\norm{T_{r,s} - T_{r,s}^{(k)}} \rightarrow 0$ as $\frac{s_n}{n}\sum\limits_{j=1}^{n}\frac{t^{n-j}}{r_j} \rightarrow 0$ as $n \rightarrow \infty$.
				Therefore, under the operator norm, the sequence of operators $\{T_{r,s}^{(k)}\}$ converges to $T_{r,s}$. Thus, $T_{r,s}\in \mathcal{K}(c_0)$ and hence $C_t$ is compact.

				\noindent We utilize a methodology similar to that of Proposition 2.2 \cite{cesaro_c0w} for the converse part. Assume that $C_t:c_0(r) \rightarrow c_0(s)$ is a compact operator. Thus, $T_{r,s}$ is also a compact operator over $c_0$. Therefore, $T_{r,s}(B_{c_0}[0,1])$ is a relatively compact set where $B_{c_0}[0,1]$ is the closed unit ball in $c_0$. Based on the result of \cite[p. 15]{compact}, there exists a sequence $y=\{y_n\} \in c_0$ such that
				\[|(T_{r,s}(x))_n| \leq |y_n| \ \ \mbox{for all} \ \ n \in \mathbb{N}, \ x \in B_{c_0}[0,1].\]
				Given that $y \in c_0,$ for every $\epsilon > 0$, there exists a $n_0 \in \mathbb{N}$ such that,
				\[|(T_{r,s}(x))_n| < \epsilon \ \ \mbox{for all} \ \ n > n_0.\]
				Let us consider the sequence $\{\zeta^{(n)}\}$ for each $n \in \mathbb{N}$, where $\{\zeta_k^{(n)}\}_{k \in \mathbb{N}} \in B_{c_0}[0,1]$ such that $\zeta_k^{(n)} = 1$ and $0$ otherwise for $k \in \{1,2, \cdots, n \}.$ Then
				\[\left|(T_{r,s}(\zeta^{(n)}))_n\right| = \frac{s_n}{n} \sum\limits_{k=1}^{n}\frac{t^{n-k}}{r_k} < \epsilon \ \ \mbox{for all} \ \ n > n_0.\]
				Hence $\left\lbrace \frac{s_n}{n} \sum\limits_{k=1}^{n}\frac{t^{n-k}}{r_k}\right\rbrace_{n} \in c_0$ and this proves the theorem.
			\end{proof}
			As $s_n$ is a bunded sequence, thus through the utilization of Theorem \ref{bt} and Theorem \ref{ct}, we derive the following sufficient conditions for the boundedness and compactness of $C_t$ within the space $B(c_0(r),c_0(s))$.
			\begin{proposition}
				\begin{itemize}
					\item [(i)] The Generalised Ces{\`a}ro operator $C_t : c_0(r) \to c_0(s)$ is bounded if
					\[\left\lbrace \sum\limits_{k=1}^{n}\frac{t^{n-k}}{r_k}\right\rbrace_{n \in \mathbb{N}} \in \ell_{\infty}. \]
					\item [(ii)]The Generalised Ces{\`a}ro operator $C_t : c_0(r) \to c_0(s)$ is compact if
					\[\left\lbrace  \sum\limits_{k=1}^{n}\frac{t^{n-k}}{r_k}\right\rbrace_{n \in \mathbb{N}} \in c_0. \]
				\end{itemize}
			\end{proposition}
			Next, we provide an example demonstrating that the bounded linear operator $C_t : c_0(r) \to c_0(s)$ is not always compact.
			\begin{example}
				Consider, Generalised Ces{\`a}ro operator $C_t : c_0(r) \to c_0(s)$, where $$r_n = \frac{1}{n},~~ \mbox{and}~~
				s_n = \left\lbrace \begin{aligned}
					1, \ & \ n \mbox{ is even }\\
					0, \ & \ n \mbox{ is odd. }
				\end{aligned}  \right.$$
				Take $\mu_n=\frac{s_n}{n} \sum\limits_{k=1}^{n}\frac{t^{n-k}}{r_k}$. In this setting, we have $\mu_{2n}=\frac{1}{2n}\sum_{k=1}^{2n}kt^{2n-k}$ and $\mu_{2n+1}=0$.
				Hence, $\lim_{n \to \infty}\mu_{2n}=\frac{1}{1-t}$ and $\lim_{n \to \infty}\mu_{2n+1}=0$. This, $C_t$ is not a compact operator but rather a bounded operator.
			\end{example}
			The following theorem provides an additional condition that is sufficient for establishing both the boundedness and compactness of $C_t$.
			\begin{theorem}
				Let $\{s_n\}$ is a sequence of positive real numbers such that $0 \leq \limsup_{n \rightarrow \infty} \frac{s_{n+1}}{s_n}<1$. Then $C_t$ is bounded and compact.
			\end{theorem}
			
			\begin{proof}
				Since $0 \leq \limsup_{n \rightarrow \infty} \frac{s_{n+1}}{s_n}<1$.
				As similar to Theorem \ref{cor11}, we have
				\begin{equation*}
					\sum\limits_{k=n_0}^{n}\frac{1}{s_k} \leq \frac{1}{(1-\mu)s_n},
				\end{equation*}
				where $0<\mu<1$ and $n_0 \in \mathbb{N}$. 
				%  there exists $0<\mu<1$ and $n_0 \in \mathbb{N}$ such that
				% \[s_{n+1} < \mu s_n \ \ \mbox{for all} \ \ n \geq n_0.\]
				% For fixed $n > n_0$ and using the above relation,
				% we obtain that
				% \begin{eqnarray*}
					% 	&&s_{n} \leq \mu^{n-k} s_k \ \ \mbox{for all} \ \ n_0 \leq k \leq n.\\
					% 	i.e., && \frac{1}{s_k} \leq \frac{\mu^{n-k}}{s_n} \ \ \mbox{for all} \ \ n_0 \leq k \leq n.
					% \end{eqnarray*}
				% Hence,
				% \begin{equation} \label{4.2.2}
					% 	\sum\limits_{k=n_0}^{n}\frac{1}{s_k} \leq \sum\limits_{k=n_0}^{n}\frac{\mu^{n-k}}{s_n}\leq \frac{1}{(1-\mu)s_n}.
					% \end{equation}
				Let $M_0 = \sum\limits_{k=1}^{n_0-1}\frac{1}{s_k}.$ Then
				\begin{equation*}
					\sum\limits_{k=1}^{n}\frac{1}{s_k} = M_0 + \sum\limits_{k=n_0}^{n}\frac{1}{s_k}.
				\end{equation*}
				Since $s_n$ is bounded and $0 \le t \le 1$, it follows that
				\begin{eqnarray*}
					\frac{s_n}{n}  \sum\limits_{k=1}^{n}\frac{t^{n-k}}{s_k}\leq \frac{s_n}{n}\sum\limits_{k=1}^{n}\frac{1}{s_k} \le \frac{1}{n}\left( M_0 \norm{s_n}_{\infty} +  \frac{1}{(1-\mu)}\right).
				\end{eqnarray*}
				%  we have 
				% 	\begin{eqnarray*}
					% 	\frac{s_n}{n}  \sum\limits_{k=1}^{n}\frac{t^{n-k}}{r_k}&\leq& \frac{M_1}{n} \left( M_0 \norm{s}_{\infty} +  \frac{1}{(1-r)}\right).
					% \end{eqnarray*}
				The required result follows from the above relation, Theorem \ref{bt} and Theorem \ref{ct}.
			\end{proof}
			\begin{theorem}
				%If $\{s_n\}$ is a sequence of positive real numbers such that $\lim \sup_{n \rightarrow \infty} \frac{s_{n+1}}{s_n} \in [0,1)$, then	
				$\sigma_p(C_t, c_0(s)) = \{\frac{1}{n}:n\in \mathbb{N}\}$ for $0 < t < 1 $.
			\end{theorem}
			\begin{proof}
				Let $C_tx=\lambda x$; $x=\{x_1,x_2,...\} \in c_0(s)$ and $x \ne 0$. Then we have 
				\begin{equation*}
					\begin{aligned}
						x_1 &=& \lambda x_1\\
						\frac{1}{2}(tx_1+x_2) &=& \lambda x_2\\
						\frac{1}{3}(t^2x_1+tx_2+x_3) &=& \lambda x_3\\
						&\vdots&\\
						\frac{1}{n}\sum_{k=1}^{n}t^{n-k}x_k&=&\lambda x_n\\
						&\vdots&
					\end{aligned}.
				\end{equation*}
				for all $n \ge 1$. If $\lambda \notin \{\frac{1}{n}:n\in \mathbb{N}\}$, then $x_n=0$ for all $n$ and $\lambda \notin \sigma_p(C_t, c_0(s))$. Thus $\sigma_p(C_t, c_0(s)) \subseteq\{\frac{1}{n}:n\in \mathbb{N}\}$.
				If $x_1 \ne 0$ then $\lambda=1$. A simple calculation yields that
				\begin{equation*}
					x_n=t^{n-1}x_1 ~\mbox{for all} ~n \ge 1.    
				\end{equation*}
				As $\{s_n\}$ is bounded strictly positive weight vector and $0< t <1$, we obtain that 
				\begin{equation*}
					\lim_{n \to \infty}x_ns_n=0.
				\end{equation*}
				Thus, $\{x_n\} \in c_0(s)$ and $1 \in \sigma_p(C_t, c_0(s))$.
				Let $x_1=0$. Therefore, we obtain
				$$\left(\lambda-\frac{1}{2}\right)x_2=0.$$
				If $x_2 \ne 0$ then $\lambda= \frac{1}{2}$. A simple calculation yields that
				\begin{equation*}
					x_n=(n-1)t^{n-2}x_2 ~ \mbox{for all}~ n \ge 2.   
				\end{equation*}
				%By using $\limsup_{n \rightarrow \infty} \frac{s_{n+1}}{s_n}<1$, we have
				
				% \begin{equation*}
					% 	\lim_{n \to \infty}\frac{x_{n+1}s_{n+1}}{x_ns_n}<t.
					% \end{equation*}
				As $\{s_n\}$ is bounded strictly positive weight vector and $0< t <1$, there exist a positive number $M$ such that $0 \le s_n \le M$, we have
				\begin{equation*}
					0 \le (n-1)t^{n-2}x_2 s_n \le M(n-1)t^{n-2}x_2.
				\end{equation*}
				Thus, $x=\{0,x_2,2tx_3,...\}\in c_0(s)$. Then $\frac{1}{2} \in \sigma_p(C_t, c_0(s)).$ If $x_m$ is the first nonzero component of the sequence $x=\{x_n\}$, then from $m$-th equation in above system of equations,
				\begin{equation*}
					\frac{1}{m}\sum_{k=1}^{n}t^{n-k}x_k=\lambda x_m,
				\end{equation*}
				we have $\frac{1}{m}x_m=\lambda x_m$, thus $\lambda =\frac{1}{m}$.
				In this case, we have 
				\begin{equation*}
					x_{m+n}=\frac{m(m+1)...(m+n-1)}{n!}t^nx_m,
				\end{equation*}
				for all $n \ge 1$. 
				%, 
				This implies,
				\begin{equation*}
					0 \le x_{m+n} s_{m+n} \le M \frac{m(m+1)...(m+n-1)}{n!}t^nx_m.
				\end{equation*}
				Let $u_n=M\frac{m(m+1)...(m+n-1)}{n!}t^nx_m.$
				By using ratio test $\lim_{n \rightarrow \infty} \frac{u_{n+1}}{u_n}=t<1$. This implies $\lim_{n \to \infty}u_n=0$, thus $\lim_{n \to \infty}x_{n+m}s_{n+m}=0$.
				We have $$x=\{0,0,0,...x_m,mtx_m,...,\frac{m(m+1)...(m+n-1)}{n!}t^nx_m,...\}\in c_0(s).$$ Therefore, $\frac{1}{m} \in \sigma_p(C_t, c_0(s)).$ This implies, $\{\frac{1}{n}:n\in \mathbb{N}\}\subseteq \sigma_p(C_t, c_0(s)).$ Hence, $\sigma_p(C_t, c_0(s)) = \{\frac{1}{n}:n\in \mathbb{N}\}$.
			\end{proof}
			
			\begin{theorem}
				If $C_t \in\mathcal{K}(c_0(s))$, then the spectrum, continuous spectrum, and residual spectrum of $C_t:c_0(s) \to c_0(s)$ are as follows,
				\begin{enumerate}
					\item[(i)] $\sigma(C_t, c_0(s))=\{\frac{1}{n}:n\in \mathbb{N}\}\cup \{0\}$,
					\item[(ii)] $\sigma_c(C_t, c_0(s))=\{0\}$,
					\item[(iii)] $\sigma_r(R_a, c_0(s)) = \emptyset.$
				\end{enumerate}
				
			\end{theorem}
			\begin{proof}
				\begin{itemize}
					\item[(i)]As $\sigma_p(C_t, c_0(s))\subseteq\sigma(C_t, c_0(s)) $ and $\sigma(C_t, c_0(s))$ is closed, we have
					\begin{equation*}
						\left\{\frac{1}{n}:n\in \mathbb{N}\right\}\cup \{0\} \subseteq\sigma(C_t, c_0(s)).
					\end{equation*}
					As $c_0(s)$ is a compact operator, all non-zero eigenvalues of a compact operator are spectral values. Thus, we have
					\begin{equation*}
						\sigma(C_t, c_0(s))\subseteq \left\{\frac{1}{n}:n\in \mathbb{N}\right\}\cup \{0\}.
					\end{equation*}
					\item[(ii)]Since $0 \notin \sigma_p(C_t, c_0(s))$,
					$C_t^{-1}$ exists.
					%Since $0 \in \sigma(C_t, c_0(s))$, and $\overline{R(C_t)}=c_0(s)$.  This implies $C_t^{-1}$ is unbounded.
					The adjoint operator $C_t^*$ of $C_t$ is given by 
					\begin{equation*}
						C_t^*=\{c_{nk}^*\} 
					\end{equation*}
					where 
					\[ c_{nk}^*= \left\lbrace \begin{aligned}
						\frac{t^{k-n}}{k}, \ & \ 1 \le n \le k\\
						0, \ & \ n > k.
					\end{aligned}  \right.\]
					Let $x=(x_1,x_2,...)\ne 0$. Then $C_t^*x=\lambda x$ gives us
					\begin{align*}
						x_1+\frac{t}{2}x_2+\frac{t^2}{3}x_3+...&=\lambda x_1\\
						\frac{1}{2}x_2+\frac{t}{3}x_3+\frac{t^2}{4}x_4+...&=\lambda x_2\\
						\frac{1}{3}x_3+\frac{t}{4}x_4+...&= \lambda x_3\\
						&\vdots
					\end{align*}
					Thus, $0 \notin \sigma_p(C_t^*,c_0^*(s)$ as if $\lambda=0$, then $x_n=0$ for all $n=1,2,...$. This implies $C_t^*$ is one-one. Thus, Theorem II.3.7 \cite[p. 59]{goldberg_unbounded} gives us, $\overline{R(C_t)}=c_0(s)$. Hence, $0 \in \sigma_c(C_t, c_0(s))$. We can obtain the required result by using the same argument as in Theorem \ref{th_spectra}(iii).
					\item[(iii)]  Since $\sigma_p(R_a, c_0(s)), \sigma_r(R_a, c_0(s))$ and $\sigma_c(R_a, c_0(s))$ forms a partition of $\sigma(R_a, c_0(s))$, we have $\sigma_r(R_a, c_0(s))=\phi$.
				\end{itemize}
				%Since $\sigma_p(C_t, c_0(s)), \sigma_r(C_t, c_0(s))$ and $\sigma_c(C_t, c_0(s))$ are disjoint and forms a partition of $\sigma(C_t, c_0(s))$. This follows required result.
				% Also, we have
				% $$x_n=\frac{1}{t^{n-1}}\frac{\left(\lambda - \frac{1}{n-1}\right)\left(\lambda - \frac{1}{n-2}\right),...(\lambda-1)}{\lambda^{n-1}}x_1.$$
				% This implies,
				% \begin{equation*}
					%     x_n=\frac{1}{t^{n-1}}\prod_{k=1}^{n-1}\left( 1-\frac{1}{k\lambda}\right)x_1 ~ \text{for all }~n \ge2,
					% \end{equation*}
				% where $x_1 \ne 0$. If $\lambda=\frac{1}{m}$ for an integer $m \ge 1$. Then $x_n=0$ for all $n \ge m+1$. Then clearly $x_n \in c_0^*(s)$. Thus, $\{\frac{1}{m}: m \in \mathbb{N}\} \subseteq \sigma_p(C_t^*,c_0^*(s)$. If $\lambda \ne \frac{1}{m}$ for all integer $m \ge 1$, then
				% \begin{equation*}
					% 			\left|\frac{{s_{n+1}x_{n}}}{{s_{n}x_{n+1}}}\right| = \left|\frac{s_{n+1}}{s_n} \frac{t}{{1-\frac{1}{n\lambda}}}\right|.
					% 		\end{equation*}
				% 		Since $\lim \sup_{n \rightarrow \infty} \frac{s_{n+1}}{s_n} \in [0,1)$, we have
				% 		\begin{equation*}
					% 			\lim\limits_{n \rightarrow \infty} \left|\frac{{s_{n+1}x_{n}}}{{s_{n}x_{n+1}}}\right|< 1.
					% 		\end{equation*}
				% 		Hence, by ratio test, $\{\frac{x_n}{s_n}\} \notin \ell_1$ and consequently $\{x_n\} \notin \ell_1(s^{-1}).$ Therefore $\sigma_p(C_t^*, c_0^*(s)) = \{\frac{1}{n}:n\in \mathbb{N},...\}$ for $0 < t < 1 $.

			\end{proof}
			Similar to Theorems \ref{Gold2} and \ref{adc}, we can derive the following results, respectively.
			\begin{theorem}
				The operator $C_t \in\mathcal{K}(c_0(s))$ satisfies the following relations.
				\begin{itemize}
					\item [(i)] $I_3 \sigma \left(C_t, c_0(s)\right)=II_3 \sigma \left(C_t, c_0(s)\right)=\phi$,
					\item [(ii)] $III_3 \sigma \left(C_t, c_0(s)\right)=\{\frac{1}{n}:n\in \mathbb{N}\}$,
					\item [(iii)] $II_2 \sigma \left(C_t, c_0(s)\right)=\{0\}$,
					\item [(iv)] $III_1 \sigma \left(C_t, c_0(s)\right)=III_2 \sigma \left(C_t, c_0(s)\right)=\phi$.
				\end{itemize}
			\end{theorem}
			
			\begin{theorem}
				The approximate point spectrum, defect spectrum, and compression spectrum of $C_t \in\mathcal{K}(c_0(s))$ over $c_0(s)$ are given by
				\begin{itemize}
					\item [(i)] $\sigma_{a p}\left(C_t, c_0(s)\right)=\{\frac{1}{n}:n\in \mathbb{N}\} \cup \{0\} $,
					\item [(ii)] $\sigma_\delta\left(C_t, c_0(s)\right)=\{\frac{1}{n}:n\in \mathbb{N}\} \cup \{0\} $,
					\item [(iii)] $\sigma_{c o}\left(C_t, c_0(s)\right)=\{\frac{1}{n}:n\in \mathbb{N}\} $.
				\end{itemize}
			\end{theorem} 
				\section{Conclusion}	%
				In conclusion, this paper investigates the spectral properties of compact Rhaly operators and discrete generalized Cesàro operators on weighted null sequence spaces. Notably, setting $t = 1$ in the discrete generalized Cesàro operator $C_t$ yields the classical Cesàro operator. Similarly, taking $a_n = \frac{1}{n}$ for $n \in \mathbb{N}$ in the Rhaly operator $R_a$ yields the classical Cesàro operator, which is widely studied in the literature. Our results establish compactness and boundedness criteria, along with spectral subdivisions, providing deeper insights into the structure of these operators in functional analysis.
				%	\bibliographystyle{splncs04}
				%	\bibliography{bib_NR_SP}
				
			\end{document}